\documentclass[12pt, reqno]{amsart}
\usepackage{amsmath, amsthm, amscd, amsfonts, amssymb, graphicx, color}
\usepackage[bookmarksnumbered, colorlinks, plainpages]{hyperref}
\hypersetup{colorlinks=true,linkcolor=red, anchorcolor=green, citecolor=cyan, urlcolor=red, filecolor=magenta, pdftoolbar=true}

\newtheorem{theorem}{Theorem}[section]

\newtheorem{proposition}[theorem]{Proposition}
\newtheorem{corollary}[theorem]{Corollary}
\theoremstyle{definition}
\newtheorem{definition}[theorem]{Definition}

\newtheorem{property}[theorem]{Property}
\theoremstyle{remark}
\newtheorem{remark}[theorem]{Remark}
\numberwithin{equation}{section}

\begin{document}

\setcounter{page}{1}

\title[Extremum principle for the Hadamard derivatives...]{Extremum principle for the Hadamard derivatives and its application to nonlinear fractional partial differential equations}

\author[M. Kirane \MakeLowercase{and} B. T. Torebek]{Mokhtar Kirane \MakeLowercase{and} Berikbol T. Torebek$^{*}$}

\address{\textcolor[rgb]{0.00,0.00,0.84}{Mokhtar Kirane\newline LaSIE, Facult\'{e} des Sciences, \newline Pole Sciences et Technologies, Universit\'{e} de La Rochelle \newline Avenue M. Crepeau, 17042 La Rochelle Cedex, France \newline NAAM Research Group, Department of Mathematics, \newline Faculty of Science, King Abdulaziz University, \newline P.O. Box 80203, Jeddah 21589, Saudi Arabia}}
\email{\textcolor[rgb]{0.00,0.00,0.84}{mkirane@univ-lr.fr}}

\address{\textcolor[rgb]{0.00,0.00,0.84}{Berikbol T. Torebek \newline Institute of
Mathematics and Mathematical Modeling \newline 125 Pushkin str.,
050010 Almaty, Kazakhstan \newline Al--Farabi Kazakh National University \newline Al--Farabi ave. 71, 050040, Almaty, Kazakhstan}}
\email{\textcolor[rgb]{0.00,0.00,0.84}{torebek@math.kz}}


\let\thefootnote\relax\footnote{$^{*}$Corresponding author}



\begin{abstract}
In this paper we obtain new estimates of the Hadamard fractional derivatives of a function at its extreme points. The extremum principle is then applied to show that the initial-boundary-value problem for linear and nonlinear time-fractional diffusion equations possesses at most one classical solution and this solution depends continuously  on the initial and boundary conditions. The extremum principle for an elliptic equation with a fractional Hadamard derivative is also proved.
\end{abstract} \maketitle

\section{Introduction}
One of the most useful and best known tools employed in the study of ordinary and partial differential equations is the extremum principle. It enables to obtain information about solutions without knowing their explicit forms.

Recently, with the development of fractional differential equations, the extremum principles for fractional differential equations have started to draw attention. This motivates us to consider the extremum principle for the Hadamard derivatives.

In \cite{Nieto}, Nieto presented two new maximum principles for a linear Riemann-Liouville fractional ordinary differential equation with initial or periodic boundary conditions. Nieto has proved the following results:

\begin{itemize}
  \item Let $\alpha\in(0, 1), \lambda \in \mathbb{R},$ and $t^{1-\alpha}u \in C([0, T])$ be such that
$$_{RL}D^{\alpha}u(t) - \lambda u(t) \geq 0,$$ $$\lim\limits_{t\rightarrow +0}t^{1-\alpha} u(t) \geq 0.$$
Then $u(t) \geq 0$ for $t \in (0, T].$
  \item Let $\alpha\in(0, 1), \lambda \in \mathbb{R},$ and $E_{\alpha,\alpha}(\lambda) < \frac{1}{\Gamma(\alpha)}.$ Suppose $t^{1-\alpha}u \in C([0, 1])$ is such that $$_{RL}D^{\alpha}u(t) - \lambda u(t) \geq 0,$$ $$\lim\limits_{t\rightarrow +0}t^{1-\alpha} u(t) = u(1),$$ where $E_{\alpha,\alpha}(\lambda)$ is a Mittag-Leffler function.
Then $u(t) \geq 0$ for $t \in (0, 1].$
\end{itemize}

Luchko \cite{Luchko1} proved an extremum principle for the Caputo fractional derivative in the following form:

\begin{itemize}
  \item Let  $f \in C^1((0, T )) \cap C([0, T ])$ attain its maximum over the interval $[0, T ]$ at the point $t_0 \in (0, T ].$
Then the Caputo fractional derivative of the function $f$ is non-negative at the point $t_0$ for any $0 < \alpha < 1:$
$_C D^\alpha f(t_0)\geq 0.$
  \item Let  $f \in C^1((0, T )) \cap C([0, T ])$ attain its minimum over the interval $[0, T ]$ at the point $t_0 \in (0, T ].$
Then the Caputo fractional derivative of the function $f$ is non-positive at the point $t_0$ for any $0 < \alpha < 1:$
$_C D^\alpha f(t_0)\leq 0.$
\end{itemize}

Based on an extremum principle for the Caputo fractional derivative he proved a maximum principle for the generalized time-fractional diffusion equation over an open bounded domain. The maximum principle was then applied to show some uniqueness and existence results for the initial-boundary-value problem for the generalized time-fractional diffusion equation \cite{Luchko2}. He also investigated the initial-boundary-value problems for the generalized multi-term time-fractional diffusion equation \cite{Luchko3} and the diffusion equation of distributed order \cite{Luchko4}, and obtained some existence results for the generalized solutions in \cite{Luchko3, Luchko4}. For the one dimensional time-fractional diffusion equation, the generalized solution to the initial-boundary value problem was shown to be a solution in the classical sense \cite{Luchko5}.

In \cite{Al-Refai} Al-Refai generalized the results of Luchko as follows:

\begin{itemize}
  \item Let $f \in C^1([0, 1])$ attain its maximum at $t_0 \in (0, 1),$ then $$_C D^\alpha f(t_0) \geq \frac{t^{-\alpha}}{\Gamma(1-\alpha)}(f(t_0) - f(0)) \geq 0,$$ for all $0 < \alpha < 1.$
  \item Let $f \in C^1([0, 1])$ attain its maximum at $t_0 \in (0, 1),$ then $$_{RL} D^\alpha f(t_0) \geq \frac{t^{-\alpha}}{\Gamma(1-\alpha)}f(t_0),$$ for all $0 < \alpha < 1.$
Moreover, if $f(t_0) \geq 0,$ then $_{RL} D^\alpha f(t_0) \geq 0.$
\end{itemize}

These inequalities were used to prove the extremum principle for linear and nonlinear time-fractional diffusion equations in works of Al-Refai and Luchko \cite{Al-Refai-Luchko1, Al-Refai-Luchko2}, and in other papers \cite{Chan, JiaLi, LiuZeng, LuchkoYa, YeLiu}.

Maximum and minimum principles for time-fractional diffusion equations with Caputo-Katugampola derivative $$_{CK}D^{\alpha,\rho}f(t)=\frac{\rho^{\alpha}}{\Gamma(1-\alpha)}\int\limits_a^t \left(t^\rho-s^\rho\right)^{-\alpha} f'(s)ds,\, \rho>0,\, 0<\alpha<1$$ are proposed in \cite{CaoKong} by Cao, Kong and Zeng. In \cite{Al-Refai-Ab, Al-Refai2, KiraneBT, Borikhanov, Ahmad} it is proved a maximum principle for the generalized time-fractional diffusion equations, based on an extremum principle for the fractional derivatives $$_{CF}D^{\alpha}f(t)=\frac{1}{1-\alpha}\int\limits_a^t \exp\left(-\frac{\alpha}{1-\alpha}(t-s)\right)f'(s)ds,\, \alpha\in (0,1)$$ and $$_{AB}D^{\alpha}f(t)=\frac{1}{1-\alpha}\int\limits_a^t E_{\alpha,1}\left(-\frac{\alpha}{1-\alpha}(t-s)^{\alpha}\right)f'(s)ds,\, \alpha\in (0,1)$$ with non-singular kernel.

An investigation of the maximum principle for fractional elliptic equations is devoted to \cite{Cabre, Capella, Cheng, Del}.

Note also that in \cite{Alikhanov, Alsaedi, Mu} using the maximum principle were obtained an upper bound of the sup norm in terms of the integral of the solution.

The present paper is devoted to the study of the extremum principle of the Hadamard derivative; some of  its  applications are presented.

The aim  of this article is:
\begin{itemize}
  \item to prove the estimates for the Hadamard and Hadamard-type fractional derivatives at the extremum points;
  \item to prove the maximum and minimum principles for the time-fractional diffusion equations with Hadamard derivative;
  \item to prove the maximum and minimum principles for the time-fractional diffusion equations with Hadamard-type derivative;
  \item to prove the uniqueness of solution and continuous dependence of a solution on the initial conditions of the initial-boundary problems for the nonlinear time-fractional diffusion equations;
  \item to prove the maximum and minimum principles for the elliptic equation with Hadamard derivative;
  \item to prove the uniqueness of solution of the boundary-value problems for the elliptic equation with fractional derivative.
\end{itemize}

\section{Some definitions and properties of fractional operators}\label{S2}

In this section, we recall some basic definitions and properties of the Hadamard fractional operators.

\begin{definition}\label{d1} \cite{Kilbas} Let $f\in L_{loc}^1([a,b]),$ where $-\infty\leq a<t<b\leq+\infty$. The Hadamard fractional integral $I^\alpha$ of order $\alpha\in\mathbb R$ ($\alpha>0$) is defined as
$$I^\alpha  f\left( t \right) = {\rm{
}}\frac{1}{{\Gamma \left( \alpha \right)}}\int\limits_a^t {\left(
\log\frac{t}{s} \right)^{\alpha  - 1} f\left( s \right)} \frac{ds}{s},$$
where $\Gamma$ denotes the Euler gamma function.
\end{definition}

\begin{definition}\cite{Kilbas} Let $f\in L^1([a,b]),$ $-\infty\leq a<t<b\leq+\infty$ and $I^{1-\alpha}f\in W^1_{2}([a,b]), 0<\alpha<1$ where $W^{1}_2([a,b])$ is the Sobolev space. The Hadamard fractional derivative $D^\alpha$ of order $\alpha$ is defined as
$$D^\alpha f \left( t \right) = t\frac{{d}}{{dt}}I^{1 - \alpha } f \left( t \right)={\rm{}}\frac{1}{{\Gamma \left( 1-\alpha \right)}}t\frac{d}{dt}\int\limits_a^t {\left(\log\frac{t}{s} \right)^{-\alpha} f\left( s \right)} \frac{ds}{s}.$$
\end{definition}

\begin{definition}\label{d2}\cite{Kilbas} Let $0<\alpha <1$ and $f\in {{W}^{1}_2}\left([ a,b ] \right).$ The Hadamard type fractional derivative of order $\alpha $ is defined by

$${D}_{*}^{\alpha } f\left( t \right)=I^{1 - \alpha } \left( t\frac{{d}}{{dt}}f \left( t \right)\right)={\rm{}}\frac{1}{{\Gamma \left( 1-\alpha \right)}}\int\limits_a^t {\left(\log\frac{t}{s} \right)^{-\alpha} f'\left( s \right)} ds.$$
\end{definition}

\begin{property}\label{p4}\cite{Kilbas}  If $f\in {{W}^{1}_2}\left([ a,b ] \right),$ then the Hadamard fractional derivative of order $\alpha $ can be represented in the form
$$D^{\alpha } f\left( t \right)=\frac{f(a)}{\Gamma(1-\alpha)}\left(\log\frac{t}{a}\right)^{-\alpha}+D_*^{\alpha } f\left( t \right).$$
\end{property}
\begin{property}\label{p5} \cite{Kilbas} If $f\in L^1([a,b])$ and $I^{1-\alpha}f\in W^1_{2}([a,b]),$ then
$${I}^{\alpha }{D}^{\alpha }f\left( t \right)=f(t)-I^{1-\alpha}f(a)\frac{\left(\log\frac{t}{a}\right)^{\alpha-1}}{\Gamma(\alpha)}.$$
\end{property}

\begin{property}\label{p6} \cite{Kilbas} If $\beta>\alpha>0,\,$ and $-\infty<a<b<+\infty,$ then
$${D}^{\alpha }\left( \log\frac{t}{a} \right)^{\beta-1}=\frac{\Gamma(\beta)}{\Gamma(\beta-\alpha)}\left( \log\frac{t}{a} \right)^{\beta-\alpha-1}.$$
\end{property}

\begin{property}\label{p7} \cite{Kilbas} If $0<\alpha<1$ and $-\infty<a<b<+\infty,$ then
${D}^{\alpha }\left( \log\frac{t}{a} \right)^{\alpha-1}=0.$
\end{property}

\section{Main Results}
\begin{proposition}\label{l7} Let a function $f\left( t \right)\in {{C}^{1}}\left([1, T ]\right)$.\\
\begin{description}
  \item[(i)] If $f\left( t \right)$ attains its maximum value over $\left[1, T \right]$ at a point ${{t}_{0}}\in \left[1, T \right]$, then for $0<\alpha <1,$ we get
\begin{equation}\label{2.1}D_{*}^{\alpha }f \left( {{t}_{0}} \right)\ge \frac{1}{\Gamma(1-\alpha) }\left(\log t_0\right)^{-\alpha}\left( f\left( {{t}_{0}} \right)-f\left(1\right) \right)\ge 0.
\end{equation}
  \item[(ii)] If $f\left( t \right)$ attains its minimum value over $\left[1, T\right]$ at a point ${{t}_{0}}\in \left[1, T\right]$, then for $0<\alpha <1,$ we have
\begin{equation}\label{2.1}D_{*}^{\alpha }f\left( {{t}_{0}} \right)\le \frac{1}{\Gamma(1-\alpha)}\left(\log t_0\right)^{-\alpha}\left( f\left( {{t}_{0}} \right)-f\left(1\right) \right)\le 0.
\end{equation}
\end{description}
\end{proposition}

\begin{proof}
For the proof of part (i), we define the auxiliary function $$g\left( t \right)=f\left( {{t}_{0}} \right)-f\left( t \right),\text{ }t\in \left[1,T \right].$$ It follows then that $g\left( t \right)\ge 0,$ on $\left[1, T \right],$ $g\left( {{t}_{0}} \right)={g}'\left( {{t}_{0}} \right)=0$ and $$D_{*}^{\alpha }g \left( t \right)=-D_{*}^{\alpha }f\left( t \right).$$
 Integrating by parts, we have
\begin{align*}D_{*}^{\alpha} g \left( {{t}_{0}} \right)&
=-\frac{1}{\Gamma(1-\alpha)}\left(\log t_0\right)^{-\alpha} g\left( 1 \right)\\&
-\frac{\alpha }{{{\Gamma\left( 1-\alpha  \right)}}}\int\limits_{1}^{{{t}_{0}}}g\left( s  \right){{\left( \log \frac{t_0}{s}  \right)}^{-\alpha -1}}\frac{ds}{s} .\end{align*}
Since $g\left( t \right)$ and $\log\frac{t_0}{s}$ are nonnegative on $\left[1, T \right]$, the integral in the last equation is nonnegative, and thus
\begin{align*}D_{*}^{\alpha} g \left( {{t}_{0}} \right)&
\leq-\frac{1}{\Gamma(1-\alpha)}\left(\log t_0\right)^{-\alpha} g\left( 1 \right)\\&=-\frac{1}{\Gamma(1-\alpha)}\left(\log t_0\right)^{-\alpha} \left(f(t_0)-f(1) \right).\end{align*}

The last inequality yields \begin{align*}-D_{*}^{\alpha} f \left( {{t}_{0}} \right)&
\leq-\frac{1}{\Gamma(1-\alpha)}\left(\log t_0\right)^{-\alpha} \left(f(t_0)-f(1) \right),\end{align*} which proves the result.
	By applying a similar argument to $-f\left( t \right)$, we obtain part (ii).
\end{proof}

\begin{proposition}\label{l8}Let  $I^{1-\alpha}f\left( t \right)\in {{C}^{1}}\left([1, T ]\right).$\\
\begin{description}
  \item[(i)] If $f\left( t \right)$ attains its maximum value over $\left[1, T \right]$ at a point ${{t}_{0}}\in \left[1, T \right]$, then for $0<\alpha <1,$ we get
\begin{equation}\label{2.1}D^{\alpha }f \left( {{t}_{0}} \right)\ge \frac{1}{\Gamma(1-\alpha)}\left(\log t_0\right)^{-\alpha} f\left( {{t}_{0}} \right).
\end{equation} Moreover, if $f(t_0) \geq 0,$ then $D^\alpha f(t_0) \geq 0.$
  \item[(ii)] If $f\left( t \right)$ attains its minimum value over $\left[1, T\right]$ at a point ${{t}_{0}}\in \left[1, T\right]$, then for $0<\alpha <1,$ we have
\begin{equation}\label{2.1}D^{\alpha }f\left( {{t}_{0}} \right)\le \frac{1}{\Gamma(1-\alpha)}\left(\log t_0\right)^{-\alpha}f\left( {{t}_{0}} \right).
\end{equation} Moreover, if $f(t_0) \leq 0,$ then $D^\alpha f(t_0) \leq 0.$
\end{description}
\end{proposition}

\begin{proof} Let a function $I^{1-\alpha}f\left( t \right)\in {{C}^{1}}\left([1, T ]\right)$ and $f\left( t \right)$ attains its maximum value over $\left[1, T \right]$ at a point ${{t}_{0}}\in \left[1, T \right].$ From Property \ref{p4}, we have $$D^{\alpha } f\left( t \right)=\frac{f(1)}{\Gamma(1-\alpha)}\left(\log{t}\right)^{-\alpha}+D_*^{\alpha } f\left( t \right).$$ Using the result in Proposition \ref{l7}, we obtain \begin{align*}D^{\alpha } f\left( t_0 \right)&=\frac{f(1)}{\Gamma(1-\alpha)}\left(\log{t_0}\right)^{-\alpha}+D_*^{\alpha } f\left( t_0 \right)\\&\geq\frac{f(1)}{\Gamma(1-\alpha)}\left(\log{t_0}\right)^{-\alpha}+\frac{1}{\Gamma(1-\alpha) }\left(\log t_0\right)^{-\alpha}\left( f\left( {{t}_{0}} \right)-f\left(1\right) \right)\\&=\frac{1}{\Gamma(1-\alpha)}\left(\log t_0\right)^{-\alpha} f\left( {{t}_{0}} \right),\end{align*} so $D^\alpha f(t_0) \geq 0$ provided $f(t_0) \geq 0.$ The part (i) of Proposition \ref{l8} is proved.

By applying a similar argument to $-f\left( t \right)$, we obtain part (ii).
\end{proof}

\section{Time-fractional diffusion equation with Hadamard type derivative}

In  this section, we  consider the nonlinear time-fractional diffusion equation
\begin{equation}\label{1.1}D_{*,t}^{\alpha }u(x,t)=\nu\Delta_x u(x,t)+F\left( x,t, u\right),\, (x,t)\in G\times \left(1,T \right]=\Omega,\end{equation}
subject to the initial condition
\begin{equation}\label{1.2}u\left( x, 1\right)=\varphi ( x ),\,x \in\bar G,\end{equation}
and boundary condition
\begin{equation}\label{1.3}u\left(x, t\right)=\psi(x,t),\, x\in\partial G,\, 1\leq t\le T,\end{equation}
where $\nu> 0,$ the functions $F\left( x,t, u\right),\varphi \left( x \right),\psi \left(x, t \right)$ are continuous and $\Delta_x$ is a Laplace operator $$\Delta_x u(x,t)=\sum\limits_{j=1}^n \frac{\partial^2 u}{\partial x_j^2}(x,t).$$ Here $G\subset\mathbb{R}^n$ is a bounded domain with smooth boundary $\partial G.$

\subsection{Maximum principle for the linear time-fractional diffusion equation}
In this subsection we shall present the extremum principle for the linear case of equation \eqref{1.1}.

\begin{theorem}\label{t1} Let $u\left( x,t \right)$  satisfy the equation \begin{equation}\label{L_eq} D_{*,t}^{\alpha }u(x,t)=\nu\Delta_x u(x,t)+F\left(x,t\right),\, (x,t)\in G\times \left(1,T \right],\end{equation} with initial-boundary conditions \eqref{1.2}-\eqref{1.3}. If $F\left( x,t \right)\geq 0$ for $\left( x,t \right)\in \overline{\Omega },$ then $$u\left( x,t \right)\ge \underset{\left( x,t \right)\in \overline{\Omega}}{\mathop{\min }}\,\{\psi \left(x, t \right),\varphi \left( x \right)\}\text{ for }\left( x,t \right)\in \overline{\Omega }.$$
\end{theorem}
\begin{proof}Let $m=\underset{\left( x,t \right)\in \overline{\Omega }}{\mathop{\min }}\,\{\psi \left(x, t \right),\varphi \left( x \right)\}\text{ }$ and $\tilde{u}\left( x,t \right)=u\left( x,t \right)-m.$ Then, from \eqref{1.2}, we obtain $$\tilde{u}\left(x,t \right)=\psi \left(x, t \right)-m\ge 0,\,\,x\in\partial G,\, t\in \left[1, T\right],$$ and $$\tilde{u}\left(x,0 \right)=\varphi \left( x \right)-m\ge 0,\,\,x\in \bar G.$$

Since $$D_{*,t}^{\alpha}\tilde{u}(x,t)=D_{*,t}^{\alpha} u(x,t)$$ and $$\Delta_x\tilde{u}\left( x,t \right)=\Delta_x u\left( x,t \right),$$
it follows that $\tilde{u}\left( x,t \right)$ satisfies \eqref{L_eq}:
$$D_{*,t}^{\alpha}\tilde{u}(x,t)=\nu\Delta_x\tilde{u}\left( x,t \right)+F\left(x,t \right),$$
and initial-boundary conditions
$$\left\{\begin{array}{l}\tilde{u}\left(x, 1\right)=\varphi \left(x\right)-m\geq 0,\,x\in \bar G,\\ \tilde{u}\left(x,t \right)=\psi\left(x, t \right)-m\ge 0,\,x\in\partial G, t\in [1,T].\end{array}\right.$$

Suppose that there exits some $\left( x,t \right) \in \overline{\Omega }$  such that $\tilde{u}\left( x,t \right)$ is negative.
Since $$\tilde{u}\left( x,t \right)\ge 0,\,\,\left( x,t \right)\in \partial G\times \left[1, T \right]\cup \bar G\times \{0\},$$
there is $\left( {{x}_{0}},{{t}_{0}} \right) \in \Omega$  such that $\tilde{u}\left( {{x}_{0}},{{t}_{0}} \right)$ is the negative minimum of
$\tilde{u}$ over $\Omega.$ It follows from Proposition \ref{l7} that

\begin{multline}\label{3.1}D_{*,t}^{\alpha }\tilde{u}\left(x_0,{t}_{0} \right)\le \frac{1}{\Gamma(1-\alpha)}\left(\log t_0\right)^{-\alpha}\left(\tilde{u}\left(x_0, {{t}_{0}} \right)-\varphi\left(x_0\right)+m \right)\\ \le \frac{1}{\Gamma(1-\alpha)}\left(\log t_0\right)^{-\alpha}\tilde{u}\left(x_0, {{t}_{0}} \right)< 0.\end{multline}

Because $\frac{\partial^2 \tilde{u}}{\partial x^2_j}(x_0, t_0)\geq 0,\, j=1,2,...,n,$ we get the inequality $$\Delta_x \tilde{u}(x_0,\, t_0)\geq 0.$$

Therefore at $\left( {{x}_{0}},{{t}_{0}} \right)$, we get $${{D}_{*,t}^{\alpha }}\tilde{u}\left(x_0, t_0 \right)< 0\,\,\, \textrm{and} \,\,\, \nu\Delta_x \tilde{u}\left( {{x}_{0}},{{t}_{0}} \right)+F\left( {{x}_{0}},{{t}_{0}} \right)\ge 0.$$ This contradiction shows that $\tilde{u}\left( x,t \right)\ge 0$ on $\overline{\Omega },$ whereupon  $u\left( x,t \right)\ge m$ on $\overline{\Omega }$.
\end{proof}

\begin{theorem}\label{t2}
Suppose that $u\left( x,t \right)$  satisfies \eqref{L_eq}, \eqref{1.2}, \eqref{1.3}. If $F\left( x,t \right)\le 0$ for $\left( x,t \right)\in \overline{\Omega },$ then $$u\left( x,t \right)\le \underset{\overline{\Omega }}{\mathop{\max }}\,\{\lambda \left( t \right),\mu \left( t \right),\varphi \left( x \right)\},\,\,\left( x,t \right)\in \overline{\Omega }.$$ \end{theorem}
Theorem \ref{t1} and \ref{t2} imply the following assertions.
\begin{corollary} Suppose that $u\left( x,t \right)$  satisfy \eqref{L_eq}, \eqref{1.2}, \eqref{1.3}.  If $F\left( x,t \right)\geq 0$ for $\left( x,t \right)\in \overline{\Omega },$ $\varphi\left(x\right)\geq 0,\,x\in\bar G,$ and $\psi\left(x, t \right)\geq 0,\,x\in \partial G,\, t\in[1,T],$ then $$u\left( x,t \right)\geq 0,\,\,\left( x,t \right)\in \overline{\Omega }.$$
\end{corollary}
\begin{corollary} Suppose that $u\left( x,t \right)$  satisfies \eqref{L_eq}, \eqref{1.2}, \eqref{1.3}. If $F\left( x,t \right)\leq 0$ for $\left( x,t \right)\in \overline{\Omega },$ $\varphi\left(x\right)\leq 0,\,x\in\bar G,$ and $\psi\left(x, t \right)\leq 0,\,x\in \partial G,\, t\in[1,T],$  then $$u\left( x,t \right)\leq 0,\,\,\left( x,t \right)\in \overline{\Omega }.$$
\end{corollary}
Theorems \ref{t1} and \ref{t2} are similar to the weak maximum principle for the heat equation. Similar to the classical case, the fractional version of the weak maximum principle can be used to prove the uniqueness of a solution.
\begin{theorem}\label{t3}
The problem \eqref{L_eq}, \eqref{1.2}, \eqref{1.3} has at most one solution.
\end{theorem}
\begin{proof}
Let ${{u}_{1}}\left( x,t \right)$ and ${{u}_{2}}\left( x,t \right)$ be two solutions of the problem \eqref{1.1}-\eqref{1.2}. Then,
$$D_{*,t}^{\alpha }\left( {{u}_{1}}\left( x,t \right)-{{u}_{2}}\left( x,t \right) \right)=\nu\Delta_x\left( {{u}_{1}}\left( x,t \right)-{{u}_{2}}\left( x,t \right) \right),$$
with zero initial and boundary conditions for ${{u}_{1}}\left( x,t \right)-{{u}_{2}}\left( x,t \right)$. It follows from Theorems \ref{t1} and \ref{t2} that $${{u}_{1}}\left( x,t \right)-{{u}_{2}}\left( x,t \right)=0,\,\, \textrm{on} \,\,\overline{\Omega }.$$ We have a contradiction. The result then follows.\end{proof}
Theorems \ref{t1} and \ref{t2} can be used to show that a solution $u\left( x,t \right)$ of the problem \eqref{L_eq}, \eqref{1.2}, \eqref{1.3} depends continuously on the initial data $\varphi \left( x \right).$

\begin{theorem}\label{t4}
Suppose $u\left( x,t \right)$ and $\overline{u}\left( x,t \right)$ are the solutions of the problem \eqref{L_eq}, \eqref{1.2}, \eqref{1.3} with homogeneous boundary conditions
corresponding to the initial data $\varphi \left( x \right)$ and $\overline{\varphi }\left( x \right),$ respectively.

If $$\underset{x\in \bar G}{\mathop{\max }}\,
\{\left| \varphi \left( x \right)-\overline{\varphi }\left( x \right) \right|\}\le \delta ,$$ then $$\left| u\left( x,t \right)-\overline{u}\left( x,t \right) \right|\le \delta.$$
\end{theorem}
\begin{proof}
The function $\widetilde{u}\left( x,t \right)=u\left( x,t \right)-\overline{u}\left( x,t \right)$ satisfies the equation $$D_{*, t}^{\alpha }\widetilde{u}\left( x,t \right)=\nu\Delta_x\widetilde{u}\left( x,t \right),$$ with initial condition $\widetilde{u}\left( x,1 \right)=\varphi \left( x \right)-\overline{\varphi }\left( x \right)$ and homogeneous boundary condition. It follows from Theorems \ref{t1} and \ref{t2} that
$$\left| \widetilde{u}\left( x,t \right) \right|\le \underset{\bar G}{\mathop{\max }}\,\{\left| \varphi \left( x \right)-\overline{\varphi }\left( x \right) \right|\}.$$
The result then follows.
\end{proof}

\subsection{Uniqueness theorems for the nonlinear time-fractional diffusion equation}\label{subsection2}
We consider the nonlinear time-fractional diffusion equation of the form \eqref{1.1},
subject to the initial and boundary conditions \eqref{1.2}, \eqref{1.3}. We start with the following uniqueness result.

\begin{theorem}\label{4.1}
If $F(x,t,u)$ is nonincreasing with respect to $u$, then the nonlinear fractional diffusion equation \eqref{1.1}, subject to the initial and boundary conditions \eqref{1.2}, \eqref{1.3}, admits at most one solution $u\in {{C}^{2}}(\bar G)\cap {{H}^{1}}((1,T])$
\end{theorem}
\begin{proof}
Assume that ${{u}_{1}}\left( x,t \right)$ and ${{u}_{2}}\left( x,t \right)$ are two solutions of the equation \eqref{1.1} subject to initial and boundary conditions \eqref{1.2}, \eqref{1.3}, and let ${v}\left( x,t \right)={{u}_{1}}\left( x,t \right)-{{u}_{2}}\left( x,t \right).$ Then ${v}\left( x,t \right)$ satisfies the equation
\begin{equation}\label{4.3}{D_{*,t}^{\alpha}{v}\left(x,t \right)}-\nu\Delta_x{{v}\left( x,t \right)}=F\left( x,t,{{u}_{2}} \right)-F\left( x,t,{{u}_{1}} \right), (x,t)\in\Omega,\end{equation} with homogeneous initial and boundary conditions \eqref{1.2}, \eqref{1.3}.
Applying the mean value theorem to $F(x,t,u)$ yields
$$F\left( x,t,{{u}_{2}} \right)-F\left( x,t,{{u}_{1}} \right)=\frac{\partial F}{\partial u}\left( {{u}^{*}} \right)\left( {{u}_{2}}-{{u}_{1}} \right)=-\frac{\partial F}{\partial u}\left( {{u}^{*}} \right)v,$$
where $\left( {{u}^{*}} \right)=(1-\mu){{u}_{1}}+\mu{{u}_{2}}$ for some $0\leq \mu \leq1$. Thus,
$${D_{*,t}^{\alpha}{v}\left(x,t \right)}-\nu\Delta_x{{v}\left( x,t \right)}=-\frac{\partial F}{\partial u}\left( {{u}^{*}} \right)v(x,t).$$
Assume by contradiction that $v$ is not identically zero. Then $v$ has either a positive
 maximum or a negative minimum. At a positive maximum $\left( {{x}_{0}},{{t}_{0}} \right)\in {\Omega }$ and as $F(x,t,u)$ is nonincreaing, we have $$\frac{\partial F}{\partial u}\left( {{u}^{*}} \right)\le 0$$ and $$-\frac{\partial F}{\partial u}\left( {{u}^{*}} \right)v\left( {{x}_{0}},{{t}_{0}} \right)\ge 0,$$ then $${D_{*,t}^{\alpha}{v}\left(x_0,t_0 \right)}-\nu\Delta_x{{v}\left( x_0,t_0 \right)}\ge 0.$$

By using results of Theorems \ref{t1} and  \ref{t2} for a positive maximum and a negative minimum, respectively, we get ${u}_{1}={u}_{2}.$
\end{proof}

\begin{theorem}\label{4.2}
Let ${{u}_{1}}\left( x,t \right)$ and ${{u}_{2}}\left( x,t \right)$ be two solutions of the time-fractional diffusion equation \eqref{1.1} that satisfy the same boundary condition \eqref{1.3} and the initial conditions ${{u}_{1}}\left( x,1 \right)={g}_{1}(x)$ and ${{u}_{2}}\left( x,1 \right)={g}_{2}(x),$ $x\in\bar G.$ If $F(x,t,u)$ is nonincreasing with respect to $u$, then it holds that
$${{\left\| {{u}_{1}}\left( x,t \right)-{{u}_{2}}\left( x,t \right) \right\|}_{C\left(\overline{\Omega }\right)}}\leq {{\left\| {{g}_{1}}\left( x \right)-{{g}_{2}}\left( x \right) \right\|}_{C(\bar G)}}.$$
\end{theorem}
\begin{proof}
Let ${v}\left( x,t \right)$=${u}_{1}(x,t)-{u}_{2}(x,t)$. Then ${v}\left( x,t \right)$ satisfies the equation
\begin{equation}\label{4.6}{D_{*,t}^{\alpha}{v}\left(x,t \right)}-\nu\Delta_x{{v}\left( x,t \right)}=-\frac{\partial F}{\partial u}\left( {{u}^{*}} \right)v(x,t), (x,t)\in\Omega,
\end{equation}
the initial condition
\begin{equation}\label{4.7}
{{v}\left( x,0 \right)}={{g}_{1}}\left( x \right)-{{g}_{2}}\left(x\right), x \in\bar G,
\end{equation}
and the homogeneous boundary condition \eqref{1.3}.
Let
$$
\mathcal{M}={{\left\| {{g}_{1}}\left( x \right)-{{g}_{2}}\left( x \right) \right\|}_{C(\bar G)}},
$$
and assume by contradiction that the result of the Theorem \ref{4.2} is not true. That is,
$$
\|u_1-u_2\|_{C(\bar{\Omega})}\nleq \mathcal{M}.
$$
Then $v$ either has a positive maximum at a point $\left( {{x}_{0}},{{t}_{0}} \right)\in {\Omega }$ with $$v\left( {{x}_{0}},{{t}_{0}} \right)=\mathcal{M}_1>\mathcal{M},$$ or it has a negative minimum at a point $\left( {{x}_{0}},{{t}_{0}} \right)\in {\Omega }$ with $$v\left( {{x}_{0}},{{t}_{0}} \right)=\mathcal{M}_2<-\mathcal{M}.$$ If $$v\left( {{x}_{0}},{{t}_{0}} \right)=\mathcal{M}_1>\mathcal{M},$$ using the initial and boundary conditions of $v$, we have $\left( {{x}_{0}},{{t}_{0}} \right)\in{{\bar\Omega }}.$

An analogous proof of those of Theorem \ref{t1} and Theorem \ref{t2} leads to $\left\| v\left( x,t \right)\right \|\leq{\mathcal{M}};$ this proves the result.
\end{proof}

\section{Time-fractional generalized diffusion equation with Hadamard derivative}

In  this section, in the bounded domain $G\subset\mathbb{R}^n$ with smooth boundary $\partial G$, we consider the time-fractional diffusion equation
\begin{equation}\label{1.1*}u_t(x,t)=\nu D_{t}^{1-\alpha}\Delta_x u(x,t)+F\left( x,t, u\right),\, (x,t)\in G\times \left(1,T \right]=\Omega,\end{equation}
with Cauchy data
\begin{equation}\label{1.2*}u\left( x, 1\right)=\varphi ( x ),\,x \in\bar G,\end{equation}
and a Dirichlet boundary condition
\begin{equation}\label{1.3*}u\left(x, t\right)=\psi(x,t),\, x\in\partial G,\, 1\leq t\le T,\end{equation}
where $\nu> 0,$ the functions $F\left( x,t, u\right),\varphi \left( x \right),\psi \left(x, t \right)$ are continuous. Here $G\subset\mathbb{R}^n$ is a bounded domain with smooth boundary $\partial G.$

\subsection{Maximum principle for the linear generalized diffusion equation with Hadamard derivative}

\vskip.3cm
In this section, we shall present the extremum principle for the linear case of equation \eqref{1.1*}.

\begin{theorem}\label{t1*} Let $u\left( x,t \right)$  satisfies the equation \begin{equation}\label{L_eq*} u_t(x,t)=\nu \Delta_x D_{t}^{1-\alpha} u(x,t)+F\left(x,t\right),\, (x,t)\in \Omega,\end{equation} with initial-boundary conditions \eqref{1.2*}-\eqref{1.3*}. If $F\left( x,t \right)\geq 0$ for $\left( x,t \right)\in \overline{\Omega },$ $\psi(x,t)\geq 0$ for $x\in\partial G,\, 1\leq t\le T$ and $\varphi ( x )\geq 0$ for $x \in\bar G,$ then $$u\left( x,t \right)\ge 0\text{ for }\left( x,t \right)\in \overline{\Omega }.$$
\end{theorem}
\begin{proof} For any $\mu\geq 0,$ let $$\tilde{u}(x, t) = u(x, t) + \mu (\log t)^{\alpha}.$$
Then,  we have
$$\tilde{u}_t(x,t)=u_t(x,t)+\mu\alpha\frac{(\log t)^{\alpha-1}}{t}, (x,t)\in\Omega,$$ $$\tilde{u}(x,1)=u(x,1)=\varphi(x),x\in\bar G,$$ $$\tilde{u}(x,t)=u(x,t)+\mu (\log t)^{\alpha}=\psi(x,t)+\mu (\log t)^{\alpha},x\in\partial G, t\in [1,T].$$
Since $\Delta_x\tilde{u}(x,t)=\Delta_x u(x,t),$ we get $$\Delta_x D_{t}^{1-\alpha}\tilde{u}(x,t)=\Delta_x D_{t}^{1-\alpha}u(x,t).$$
Hence, $\tilde{u}(x, t)$ satisfies the equation
\begin{align*}\tilde{u}_t(x,t)&=\nu\Delta_x D_{t}^{1-\alpha}\tilde{u}(x,t)+F\left(x,t\right)\\&+\mu\alpha\frac{(\log t)^{\alpha-1}}{t},\, (x,t)\in G\times \left(1,T \right].\end{align*}

Suppose that there exists some $\tilde{u}(x, t) \in \bar\Omega$ such that $\tilde{u}(x, t) < 0.$ Since $$\tilde{u}(x, t) \geq 0,\, (x, t) \in \partial G\times[1, T]\cup\bar G\times{0},$$ there is $(x_0, t_0) \in G\times(1, T]$ such that $\tilde{u}(x_0, t_0)$ is the negative minimum of $\tilde{u}$ over $\bar\Omega.$

It follows from Proposition \ref{l8} (ii) that

\begin{equation}\label{3.1*}D_{t}^{1-\alpha }\tilde{u}\left(x_0,{t}_{0} \right)\le \frac{1}{\Gamma(\alpha)}\left(\log t_0\right)^{\alpha-1}\tilde{u}\left(x_0, {{t}_{0}} \right) < 0.\end{equation}

Let $w\left( x,t \right)=D_{t}^{1-\alpha }\tilde{u}\left( x,t \right).$
Since $\tilde{u}\left( x,t \right)$ is bounded in $\overline{\Omega },$ then we have
\begin{equation}\label{3.2*}
I_t^\alpha u\left(x, t \right) = {\rm{
}}\frac{1}{{\Gamma \left( \alpha \right)}}\int\limits_1^t {\left(
\log\frac{t}{s} \right)^{\alpha  - 1} u\left(x, s \right)} \frac{ds}{s} \to 0\text{ as }t\to 1.
\end{equation}
It follows from Property \ref{p6} that
$$D_{t}^{\alpha}w\left( x,t \right)=D_{t}^{\alpha}D_{t}^{1-\alpha }\tilde{u}\left( x,t \right)=\tilde{u}_t(x,t).$$
From Property \ref{p7}, for any $t > 1$, we get
\begin{align*}D_{t}^{1-\alpha }\tilde{u}\left(x,t \right)&=D_{t}^{1-\alpha }u\left( x,t \right)+\mu D_{t}^{1-\alpha}(\log t)^{\alpha} \\&= D_{t}^{1-\alpha }u\left( x,t \right)+\mu\frac{\Gamma(\alpha+1)}{\Gamma(2\alpha)} (\log t)^{2\alpha-1}.\end{align*}
It follows from Property \ref{p4} that
\begin{equation}\label{4.7}D_t^{1-\alpha } u\left(x, t \right)=\frac{1}{\Gamma(1-\alpha)}\varphi(x)\left(\log t\right)^{\alpha-1}+D_{*,t}^{\alpha } u\left(x, t \right)\,\,\,\textrm{for}\,\,\,t>1.\end{equation}
Since the left-hand side of \eqref{4.7} and the first term of the right-hand side of \eqref{4.7} exist, it
follows that the second term on the right-hand side exists and tends to $0$ as $t \rightarrow 1.$ As
$t \rightarrow 1,$ $\varphi(x)\left(\log t\right)^{\alpha-1}\geq 0.$ Therefore, $$D_t^{1-\alpha}u(x, t) > 0\,\,\, \textrm{when} \,\,\,t \rightarrow 1.$$ Hence, we obtain
\begin{align*}w(x,t)=D_t^{1-\alpha}\tilde{u}(x,t)=D_{t}^{1-\alpha }u\left( x,t \right)+\mu\frac{\Gamma(\alpha+1)}{\Gamma(2\alpha)} (\log t)^{2\alpha-1}.
\end{align*}
Furthermore, it follows from the boundary condition of $\tilde{u}(x, t)$ that
\begin{align*}w(x,t)&=D_t^{1-\alpha}\tilde{u}(x,t)=D_t^{1-\alpha}\psi(x,t)\\&+\mu\frac{\Gamma(\alpha+1)}{\Gamma(2\alpha)} (\log t)^{2\alpha-1}\geq 0,\, t\geq 1.
\end{align*}
Therefore, $w(x, t)$ satisfies the problem
\begin{align*}&D_t^\alpha w(x,t)=\nu\Delta_x w(x,t)+\tilde{F}(x,t),\,(x,t)\in\Omega, \\& w(x,1)\geq 0, x\in \bar G, \\& w(x,t)\geq 0,\,x\in\partial G,\,1\leq t\leq T,\end{align*}
where $\tilde{F}(x,t)=F(x,t)+\mu\alpha (\log t)^{\alpha-1}/t.$

From \eqref{3.1*}, we have $w(x_0, t_0) < 0.$ Since $w(x, t) \geq 0$ on the boundary, there exists $(x_*, t_*) \in \Omega$ such that $w(x_*, t_*)$ is the negative minimum of function $w(x, t)$ in $\bar\Omega.$  It follows from Proposition \eqref{l8} part (ii) that
$$D^{\alpha }w\left(x_*, {{t}_{*}} \right)\le \frac{1}{\Gamma(1-\alpha)}\left(\log t_*\right)^{-\alpha}w\left(x_*, {{t}_{*}} \right)<0.$$
Since $w(x_*, t_*)$ is a local minimum, we obtain $\frac{\partial^2 w}{\partial x^2_j}w(x_*, t_*) \geq 0,\, j=1,2,...,n,$ and $\Delta_x w(x_*,\, t_*)\geq 0.$

Therefore at $\left( {{x}_{*}},{{t}_{*}} \right)$, we get $${{D}_{*,t}^{\alpha }}w\left(x_*, t_* \right)< 0\,\,\, \textrm{and} \,\,\,\nu\Delta_x w\left( {{x}_{*}},{{t}_{*}} \right)+F\left( {{x}_{*}},{{t}_{*}} \right)\ge 0.$$ This contradiction shows that $w\left( x,t \right)\ge 0$ on $\overline{\Omega },$ whereupon  $$u\left( x,t \right)\ge -\mu (\log t)^{\alpha}\,\,\, \textrm{on} \,\,\,\overline{\Omega }$$  for any $\mu.$ Since $\mu$ is arbitrary, we have $u(x, t) \geq 0$ on $\bar\Omega.$
\end{proof}
A similar result can be obtained for the nonpositivity of the solution $u(x, t)$ by considering $-u(x, t)$ when $F(x,t)\leq 0,$ $\varphi(x) \leq 0$ and $\psi(x, t) \leq 0.$

\begin{theorem}\label{t2*}
Let $u\left( x,t \right)$  satisfy the equation \eqref{L_eq*} with initial-boundary conditions \eqref{1.2*}-\eqref{1.3*}. If $F\left( x,t \right)\leq 0$ for $\left( x,t \right)\in \overline{\Omega },$ $\psi(x,t)\leq 0$ for $x\in\partial G,\, 1\leq t\le T$ and $\varphi ( x )\leq 0$ for $x \in\bar G,$ then $$u\left( x,t \right)\le 0\text{ for }\left( x,t \right)\in \overline{\Omega }.$$ \end{theorem}
Theorem \ref{t1*} and \ref{t2*} imply the following assertions.
\begin{theorem}\label{t1**} Suppose that $u\left( x,t \right)$  satisfies \eqref{L_eq*}, \eqref{1.2*}, \eqref{1.3*} If $F\left( x,t \right)\geq 0$ for $\left( x,t \right)\in \overline{\Omega },$ then $$u\left( x,t \right)\geq \min\limits_{(x,t)\in\bar\Omega}\left\{\varphi(x),\psi(x,t)\right\},\,\,\left( x,t \right)\in \overline{\Omega }.$$
\end{theorem}
\begin{proof} Let $m =\min\limits_{(x,t)\in\bar\Omega}\left\{\varphi(x),\psi(x,t)\right\}$ and $\tilde{u}(x,t)=u(x,t)-m.$ Then, $$\tilde{u}(x,1)=\varphi(x)-m\geq 0,\,\,x\in\bar G,$$ $$\tilde{u}(x,x)=\psi(x,t)-m\geq 0,\,\, x\in\partial G,\, 1\leq t\leq T.$$ Since $$\tilde{u}_t(x,t)=u_t(x,t),$$ $$\Delta_x D_t^{1-\alpha}\tilde{u}(x,t)=\Delta_x D_t^{1-\alpha}u(x,t),$$ it follows that $u(x, t)$ satisfies \eqref{L_eq*}. Thus, it follows from an argument similar to the proof of Theorem \ref{t1*} that $$\tilde{u}(x,t)\geq 0,\,(x,t)\in\bar\Omega.$$ That is, $$u\left( x,t \right)\geq \min\limits_{(x,t)\in\bar\Omega}\left\{\varphi(x),\psi(x,t)\right\},\,\,\left( x,t \right)\in \overline{\Omega }.$$ The theorem \ref{t1**} is proved.
\end{proof}
A similar result can be obtained for the nonpositivity of the solution $u(x, t)$ by considering $-u(x, t).$
\begin{theorem} Suppose that $u\left( x,t \right)$  satisfies \eqref{L_eq*}, \eqref{1.2*}, \eqref{1.3*}. If $F\left( x,t \right)\leq 0$ for $\left( x,t \right)\in \overline{\Omega },$ then $$u\left( x,t \right)\leq \max\limits_{(x,t)\in\bar\Omega}\left\{\varphi(x),\psi(x,t)\right\},\,\,\left( x,t \right)\in \overline{\Omega }.$$
\end{theorem}

\subsection{Uniqueness results for the Linear and Nonlinear time-fractional diffusion equation}
\vskip.3cm
The maximum principle for the time-fractional diffusion equation \eqref{L_eq*} can be used to prove the uniqueness of a solution.
\begin{theorem}\label{t3*}
The problem \eqref{L_eq*}, \eqref{1.2*}, \eqref{1.3*} has at most one solution.
\end{theorem}
\begin{proof}
Let ${{u}_{1}}\left( x,t \right)$ and ${{u}_{2}}\left( x,t \right)$ be two solutions of the initial-boundary value problem \eqref{L_eq*}, \eqref{1.2*}, \eqref{1.3*} and $\hat{u}(x,t)={{u}_{1}}\left( x,t \right)-{{u}_{2}}\left( x,t \right)$. Then,
$$\hat{u}_t\left( x,t \right)=\nu\Delta_x D_{t}^{1-\alpha}\hat{u}\left( x,t \right),$$
with homogeneous initial and boundary conditions \eqref{1.2*}, \eqref{1.3*} for $\hat{u}\left( x,t \right)$. It follows from Theorems \ref{t1*} and \ref{t2*} that $\hat{u}\left( x,t \right)=0$ on $\overline{\Omega}.$ Consequently ${{u}_{1}}\left( x,t \right)={{u}_{2}}\left( x,t \right).$ The result then follows.\end{proof}
Theorems \ref{t1*} and \ref{t2*} can be used to show that a solution $u\left( x,t \right)$ of the problem \eqref{L_eq*}, \eqref{1.2*}, \eqref{1.3*} depends continuously on the initial data $\varphi \left( x \right).$

\begin{theorem}\label{t4*}
Suppose $u\left( x,t \right)$ and $\bar{u}\left( x,t \right)$ are the solutions of the equation \eqref{L_eq*} that satisfy the same boundary condition \eqref{1.3*} and the initial conditions $u\left( x,1 \right)=\varphi(x)$ and $\bar{u}\left( x,1 \right)=\bar{\varphi}(x),$ $x\in\bar G.$

If $$\underset{x\in \bar G}{\mathop{\max }}\,
\{\left| \varphi \left( x \right)-\bar{\varphi }\left( x \right) \right|\}\le \delta,$$ then $\left| u\left( x,t \right)-\bar{u}\left( x,t \right) \right|\le \delta.$
\end{theorem}
\begin{proof}
The function $\tilde{u}\left( x,t \right)=u\left( x,t \right)-\bar{u}\left( x,t \right)$ satisfies the equation $$\tilde{u}_t\left( x,t \right)=\nu\Delta_x D_{t}^{1-\alpha}\tilde{u}\left( x,t \right),$$ with initial condition $\tilde{u}\left(x, 1\right)=\varphi \left( x \right)-\bar{\varphi }\left( x \right)$ and boundary condition \eqref{1.3*}. It follows from Theorems \ref{t1*} and \ref{t2*} that
$$\left| \tilde{u}\left( x,t \right) \right|\le \underset{\bar G}{\mathop{\max }}\,\{\left| \varphi \left( x \right)-\bar{\varphi }\left( x \right) \right|\}.$$
The result then follows.
\end{proof}

\begin{theorem}\label{4.1*}
If $F(x,t,u)$ is nonincreasing with respect to $u$, then the nonlinear fractional diffusion equation \eqref{1.1*} subject to the initial and boundary conditions \eqref{1.2*}, \eqref{1.3*} admits at most one solution $u\in {{C}^{2}}(\bar G)\cap {{H}^{1}}((1,T])$
\end{theorem}
\begin{theorem}\label{4.2*}
If ${{u}_{1}}\left( x,t \right)$ and ${{u}_{2}}\left( x,t \right)$ are two solutions of the time-fractional diffusion equation \eqref{1.1*} that satisfy the same boundary condition \eqref{1.3*} and the initial conditions $${{u}_{1}}\left( x,1 \right)={g}_{1}(x)\,\,\, \textrm{and} \,\,\,{{u}_{2}}\left( x,1 \right)={g}_{2}(x), x\in\bar G$$ and if $F(x,t,u)$ is nonincreasing with respect to $u$, then it holds that
$${{\left\| {{u}_{1}}\left( x,t \right)-{{u}_{2}}\left( x,t \right) \right\|}_{C\left(\overline{\Omega }\right)}}\leq {{\left\| {{g}_{1}}\left( x \right)-{{g}_{2}}\left( x \right) \right\|}_{C(\bar G)}}.$$
\end{theorem}
Theorems \ref{4.1*} and \ref{4.2*} are proved similarly as Theorems \ref{4.1} and \ref{4.2} in Section \ref{subsection2}.

\section{Maximum principle for an elliptic equation with fractional derivative in a multidimensional cube and its applications}
In this section, we consider an elliptic equation with the Hadamard derivative in a multidimensional cube
\begin{multline}\label{6.1} \Delta_x u(x)+\sum\limits_{j=1}^n a_j(x) \frac{\partial u}{\partial x_j}(x)+\sum\limits_{j=1}^n b_{j}(x) D_{x_{j}}^{\alpha}u(x)\\+c(x)u(x)=F(x),\,x\in \prod\limits_{j=1}^n (1, h_j)=\Omega,\end{multline}
where $a_j(x), b_j(x), c(x)$ and $F(x)$ are given functions.

\subsection{Weak and strong maximum principle}
\vskip.3cm
We start with a weak maximum principle that is formulated in the following theorem.
\begin{theorem}\label{t6.1} Let a function $u(x)\in C^2(\Omega)\cap C^1(\bar\Omega)$ satisfy the equation \eqref{6.1} and $b_j(x)<0,\,c(x)\leq 0,\,x\in\bar\Omega.$ If $F(x)\geq 0,$ then the inequality \begin{equation}\label{6.2}\max\limits_{x\in\bar\Omega} u(x)\leq\max\limits_{x\in\partial\Omega} \{u(x),0\}\end{equation} holds true, where $\partial\Omega$ is a boundary of $\Omega.$
\end{theorem}
\begin{proof} To prove the theorem, let us assume that the inequality \eqref{6.2} does not hold true under the conditions that are formulated in Theorem \ref{t6.1}, i.e. that the function $u(x)$ attains its positive maximum, say $M > 0$ at a point $x^*=(x_1^*,...,x_n^*) \in \Omega.$

Because $$c(x^*)\leq 0, \frac{\partial u}{\partial x_j}(x^*)=0\,\,\, \textrm{and} \,\,\,\frac{\partial^2 u}{\partial x^2_j}(x^*)\leq 0,$$ we first get the inequality
$$\Delta_x u(x^*)+c(x^*)u(x^*)\leq 0.$$
Then, it follows from Proposition \ref{l8} that
$$D_{x_j}^{\alpha }u \left( {{x}^*} \right)\ge \frac{1}{\Gamma(1-\alpha)}\left(\log x_j^*\right)^{-\alpha} u\left( x^* \right)>0.$$
As $b_{j}(x^*)<0,$ then we obtain
$$\sum\limits_{j=1}^n b_{j}(x^*) D_{x_{j}}^{\alpha}u(x^*)<0.$$
The last two inequalities lead to the inequality
$$\Delta_x u(x^*)+\sum\limits_{j=1}^n b_{j}(x^*) D_{x_{j}}^{\alpha}u(x^*)+c(x^*)u(x^*)<0$$ that contradicts the following one: \begin{align*}\Delta_x u(x)&+\sum\limits_{j=1}^n a_j(x) \frac{\partial u}{\partial x_j}(x)\\& +\sum\limits_{j=1}^n b_{j}(x) D_{x_{j}}^{\alpha}u(x)+c(x)u(x)\geq 0,\,x\in\Omega\end{align*} of Theorem \ref{t6.1}. A contradiction. The theorem is proved.
\end{proof}
The following theorem is proved similarly.
\begin{theorem}\label{t6.2} Let a function $u(x)\in C^2(\Omega)\cap C^1(\bar\Omega)$ satisfy the equation \eqref{6.1} and $b_j(x)>0,\,c(x)\leq 0,\,x\in\bar\Omega.$ If $F(x)\leq 0,$ then the inequality \begin{equation}\label{6.3}\min\limits_{x\in\bar\Omega} u(x)\geq\min\limits_{x\in\partial\Omega} \{u(x),0\}\end{equation} holds true.
\end{theorem}
\begin{remark}\label{rem1} In the proof of the weak maximum principle, we have in fact deduced a statement that is stronger than the inequality \eqref{6.2}, namely, we proved that a function u that fulfills the conditions of Theorem \ref{t6.1} cannot attain its positive maximum at a point $x^* \in \Omega.$\end{remark}
The statement of Remark \ref{rem1} is now employed to derive a strong maximum principle for the elliptic equation \eqref{6.1}.

\begin{theorem}\label{t6.3} Let a function $u(x)\in C^2(\Omega)\cap C^1(\bar\Omega)$ satisfy the homogeneous elliptic equation \eqref{6.1} and $c(x) \leq 0, x\in\Omega.$
If the function u attains its maximum and its minimum at some points that belong to $\Omega,$ then it is a constant, more precisely $$u(x) =0,\,x\in \Omega.$$\end{theorem}

\begin{proof} The proof of Theorem is a very simple one. Indeed, according to Remark \ref{rem1}, $$u(x) \leq 0, x\in\bar\Omega$$ for a function $u(x)$ that attains its maximum at a point $x^*\in\Omega.$
Now let us consider the function $-u(x)$ that satisfies the homogeneous equation \eqref{6.1} and possesses a maximum at the minimum point of $u(x)$ and thus at a point that belongs to $\Omega.$ The
maximum of $-u(x)$ cannot be positive according to Remark \ref{rem1} and we get the inequality $$-u(x) \leq 0, x\in\Omega.$$ Putting the two last inequalities together, we get the statement of Theorem \ref{t6.3}.
\end{proof}
\subsection{Applications of the maximum principles}
In this section, we start with the boundary-value problem
\begin{equation}\label{6.4}u(x)=\varphi(x),\,x\in\partial\Omega,\end{equation} for the elliptic equation \eqref{6.1}.
The following result is a direct consequence of the weak maximum principle

\begin{theorem} Let $F(x),a_j(x), b_j(x), \varphi(x)$ and $c(x)\leq 0$ be smooth functions. Then the boundary-value problem \eqref{6.4} for the equation \eqref{6.1} possesses at most one solution $u(x)$ in the functional space $C^2(\Omega)\cap C^1(\bar\Omega).$
\end{theorem}

The following two theorems follow directly from Theorem \ref{t6.1} and Theorem \ref{t6.2}.
\begin{theorem}\label{t6.4} Let $u(x)\in C^2(\Omega)\cap C^1(\bar\Omega)$ fulfill the equation \eqref{6.1} and $c(x) \leq 0, x\in\Omega.$ If $u(x)$ satisfies the boundary condition \eqref{6.4} and $\varphi(x) \geq 0, x\in \partial\Omega,$ then $$u(x) \geq 0, x\in\bar\Omega.$$
\end{theorem}

\begin{theorem}\label{t6.5} Let $u(x)\in C^2(\Omega)\cap C^1(\bar\Omega)$ is the solution of elliptic equation \eqref{6.1} and $c(x) \leq 0, x\in\Omega.$ If $u(x)$ satisfies the boundary condition \eqref{6.4} and $\varphi(x) \leq 0, x\in \partial\Omega,$ then $$u(x) \leq 0, x\in\bar\Omega.$$
\end{theorem}

We now consider a non-linear elliptic equation in the form
\begin{multline}\label{6.5} \Delta_x u(x)+\sum\limits_{j=1}^n a_j(x) \frac{\partial u}{\partial x_j}(x)+\sum\limits_{j=1}^n b_{j}(x) D_{x_{j}}^{\alpha}u(x)\\+c(x)u(x)=F(u,x),\,x\in \prod\limits_{j=1}^n (1, h_j)=\Omega.\end{multline}
Under some suitable conditions on the non-linear part $F(u,x),$ the maximum principle for the elliptic equation \eqref{6.5} leads to an uniqueness result for the boundary-value problem \eqref{6.4} for the equation \eqref{6.5}.

\begin{theorem}\label{t6.6} Let $F(u,x)$ be a smooth and non-increasing function with respect to the variable $u.$ Then the boundary-value
problem \eqref{6.5}, \eqref{6.4} possesses at most one solution $u(x)\in C^2(\Omega)\cap C^1(\bar\Omega).$
\end{theorem}

\begin{proof} Again we employ a proof by contradiction and first suppose that $u_1(x)$ and $u_2(x)$ are two solutions of boundary-value
problem \eqref{6.5}, \eqref{6.4} that belong to the functional space $C^2(\Omega)\cap C^1(\bar\Omega).$  Then the auxiliary function $u(x) = u_1(x)-u_2(x)$ satisfies the equation
\begin{multline*} \Delta_x u(x)+\sum\limits_{j=1}^n a_j(x) \frac{\partial u}{\partial x_j}(x)+\sum\limits_{j=1}^n b_{j}(x) D_{x_{j}}^{\alpha}u(x)\\+c(x)u(x)=F(u_1,x)-F(u_2,x),\,x\in \Omega\end{multline*} and the homogeneous boundary condition \eqref{6.4}, i.e.,
$$u(x)=0,\,x\in\partial\Omega.$$ Applying the mean value theorem to the smooth function $F(u,x)$ yields the equation
\begin{multline*} \Delta_x u(x)+\sum\limits_{j=1}^n a_j(x) \frac{\partial u}{\partial x_j}(x)+\sum\limits_{j=1}^n b_{j}(x) D_{x_{j}}^{\alpha}u(x)\\+c(x)u(x)=\frac{\partial F}{\partial u}(u^*)u(x),\,x\in \Omega,\end{multline*} where $u^*(x)=(1-\mu)u_1(x)+\mu u_2(x),$ for some $\mu\in[0,1]$ that can be rewritten in the form
\begin{multline*} \Delta_x u(x)+\sum\limits_{j=1}^n a_j(x) \frac{\partial u}{\partial x_j}(x)+\sum\limits_{j=1}^n b_{j}(x) D_{x_{j}}^{\alpha}u(x)\\+c(x)u(x)-\frac{\partial F}{\partial u}(u^*)u(x)=0,\,x\in \Omega.\end{multline*}
Because $F(u,x)$ is a non-increasing function with respect to the variable $u,$ we get the inequality $$\frac{\partial F}{\partial u}(u^*)\leq 0,\,x\in\Omega.$$ Now we are in a position to apply Theorem \ref{t6.4} that produces the inequality $$u(x) \leq 0,\, x\in\bar\Omega.$$
The reasoning we employed for the function $u(x)$ is also valid for the auxiliary function $-u(x) = u_2(x) - u_1(x)$ that leads to the inequality $$-u(x) \leq 0,\, x\in\bar\Omega.$$ Combining the last two inequalities, we arrive at the formula $$u(x) =0,\, x\in\bar\Omega$$ that means that any two solutions of the boundary-value problem \eqref{6.5}, \eqref{6.4} coincide and thus the statement of the theorem is proved.
\end{proof}

\section*{Acknowledgements} The second author was financially supported by a grant No.AP05131756 from the Ministry of Science and Education of the Republic of Kazakhstan.

\end{document}